\documentclass[fontsize=11pt,a4paper,DIV13]{scrartcl}

\usepackage[utf8]{inputenc} 
\usepackage{cmbright}
\usepackage{url}

\newcommand{\changefont}[3]{
\fontfamily{#1}\fontseries{#2}\fontshape{#3}\selectfont}
\def\SC{\changefont{cmr}{m}{sc}}

\usepackage{graphicx}
\input{psfig.sty}

\usepackage{amssymb}
\usepackage{amsmath,amsthm}

\usepackage{tikz,ifthen}

\def\comment#1{}

\def\ni{\noindent}

\def\ITEMMACRO #1 ??? #2 ???{\par\medskip\noindent%
\hangindent=#2em\setbox0\hbox{#1 \kern5pt}%
\ifdim\wd0<\hangindent\setbox0\hbox to\hangindent{\hss#1\quad}\fi%
\box0\ignorespaces}

\def\Item(#1){\ITEMMACRO {\rm (#1)} ??? 1.8 ???}
\def\FreeItem#1{\ITEMMACRO {#1} ??? 1.8 ???}

\let\Bitem=\bItem
\def\BrackItem[#1]{\ITEMMACRO [#1] ??? 1.8 ???}

\def\Note{\ITEMMACRO {\NI\bf Note.}  ??? 1.8 ??? \\ \bgroup\sf}
\def\EndNote{\par\egroup\medskip\ni}

\newtheorem{theorem}{Theorem}[section]
\newtheorem{lemma}[theorem]{Lemma}
\newtheorem{proposition}[theorem]{Proposition}
\newtheorem{conjecture}{Conjecture}
\newtheorem{corollary}[theorem]{Corollary}
\newtheorem{observation}[theorem]{Observation}
\newtheorem{remark}[theorem]{Remark}
\theoremstyle{definition}

\def\Fact#1.{\par\smallskip{\ni\bf Fact~#1.}\ }
\def\Problem#1.{\par\smallskip{\ni\bf Problem~#1.}\ }
\def\Claim#1.{\medbreak\ni{\bf Claim~#1.}\ }
\def\Case#1.{\medbreak\ni{\bf Case~#1.}\ }
\def\SubCase#1.{\medbreak\ni{\bf Subcase~#1.}\ }

\definecolor{RED}{rgb}{.84,0,0}
\definecolor{BLUE}{rgb}{0,0,.75}

\def\ovl{\overline}

\def\RR{\hbox{\sf I\kern-1ptR}}
\def\NN{\hbox{\sf I\kern-1ptN}}
\def\ZZ{\hbox{\sf Z\kern-4ptZ}}

\begin{document}
%

\title{On the Duality of Semiantichains and Unichain Coverings}

\author{
\parbox{6.5cm}{\center
{\SC Bart{\l}omiej Bosek
\footnote{Partially supported by NCN grant MO-2011/03/B/ST6/01367.}}\\[3pt]
\normalsize
\small
        {Theoretical Computer Science Department\\
  Faculty of Math. and Comp. Sci.\\
  Jagiellonian University\\
  {\L}ojasiewicza 6, 30-348 Krak\'{o}w, Poland}}
\and
\parbox{6.5cm}{\center
{\SC Stefan Felsner
\footnote{Partially supported by DFG grant FE-340/7-2 and 
              ESF EuroGIGA project Compose.}}\\[3pt]
\normalsize
\small
        {Institut f\"ur Mathematik\\
         Technische Universit\"at Berlin\\
         Strasse des 17. Juni 136\\
         D-10623 Berlin, Germany}}
\and
\parbox{6.5cm}{\center
{\SC Kolja Knauer
\footnote{Partially supported by 
              ESF EuroGIGA project GraDr and ANR TEOMATRO grant ANR-10-BLAN 0207.\hfil\break
A conference version of this paper appeared in 
Proc.~CSR 2012, LNCS 7353, 43--51, 2012.}}\\[3pt]
\normalsize
\small
        {Institut de Math\'ematiques (I3M)\\
         Universit\'e Montpellier 2\\
         Place Eugene Bataillon\\
         34000 Montpellier, France}}
\and
\parbox{6.5cm}{\center
{\SC Grzegorz Matecki
\footnotemark[1]}\\[3pt]
\normalsize
\small
        {Theoretical Computer Science Department\\
  Faculty of Math. and Comp. Sci.\\
  Jagiellonian University\\
  {\L}ojasiewicza 6, 30-348 Krak\'{o}w, Poland}}
}

\date{}
\maketitle

\vskip-0mm
\begingroup\fontsize{10}{12}\rm
\centerline{\normalsize\bf Abstract}
\bigskip
\noindent 
We study a min-max relation  conjectured by Saks and West:
For any two posets $P$ and $Q$ the size of a maximum semiantichain and the
size of a minimum unichain covering in the product $P\times Q$ are equal.
For positive we state conditions on $P$ and $Q$ that imply
the min-max relation. Based on these conditions we identify some new
families of posets where the conjecture holds and get
easy proofs for several instances where the conjecture had been verified before.
However, we also have examples showing that in
general the min-max relation is false, i.e., we disprove the
Saks-West conjecture.

\vskip10pt\ni
{\bf Mathematics Subject Classifications (2010) 
  06A07,    
  05B40,    
  90C46.    
}\endgroup

\section{Introduction}
This paper is about min-max relations with respect to chains and
antichains in posets. In a poset, \emph{chains} and \emph{antichains}
are sets of pairwise comparable and pairwise incomparable elements,
respectively.  By the \emph{height} $h(P)$ and the \emph{width} $w(P)$
of poset $P$ we mean the size of a largest chain and a largest
antichain in $P$, respectively. The product $P\times Q$ of two posets $P$ and
$Q$ is an order defined on the product of their underlying sets by
$(u, x)\leqslant_{P\times Q} (v, y)$ if and only if $u\leqslant_P v$ and $x \leqslant_Q y$.

Dilworth~\cite{Dil-50} proved that any poset $P$ can be covered with a
collection of $w(P)$ chains. Greene and Kleitman~\cite{Gre-76a}
generalized Dilworth's Theorem. A \emph{$k$-antichain} in $P$ is a
subset of $P$ which may be decomposed into $k$ antichains. We denote
the size of a maximal $k$-antichain of $P$ by $d_k(P)$ or simply $d_k$
if the poset is unambiguous from the context. The theorem of Greene
and Kleitman says that for every $k$ there is a chain-partition
$\mathcal{C}$ of $P$ such that
$d_k(P)=\sum_{C\in\mathcal{C}}\min(k,|C|)$.  In~\cite{Sak-79} Saks
proves the theorem of Greene and Kleitman by showing the following
equivalent statement:

\begin{theorem}\label{thm:gk}
  In a product $C\times Q$ where $C$ is a chain, the size of a maximum
  antichain $A$ equals the size of a minimum chain covering with
  chains of the form $\{c\}\times C'$ and $C\times\{q\}$.  In
  particular this number is $d_{|C|}(Q)$.
\end{theorem}

The \emph{Saks-West Conjecture} is about a generalization of
Theorem~\ref{thm:gk}. A chain in a product $P\times Q$ is a
\emph{unichain} if it is of the form $\{p\}\times C'$ or
$C\times\{q\}$. A \emph{semiantichain} is a set $S\subseteq P\times Q$
such that no two distinct elements of $S$ are contained in an
unichain. With this notation we are ready to state the Saks-West 
conjecture:
\begin{conjecture}\label{conj:sw}
In every product $P\times Q$ of two posets the size of a largest semiantichain
equals the size of a smallest unichain covering.
\end{conjecture}
The conjecture had already been around for a while when it appeared in
print~\cite{WeSa-82}. Theorem~\ref{thm:gk} deals with the special case
of conjecture where one of $P$ and $Q$ is a
chain. Several partial results mostly regarding special classes of
posets that satisfy the conjecture have been verified.
\begin{itemize}
\item Tovey and West~\cite{Tov-81} relate the problem to dual pairs of
  integer programs of packing and covering type and explain the
  interpretation as independence number and clique covering in the
  product of perfect graphs. Furthermore, they take first steps
  towards the investigation of posets with special chain- and
  antichain-decomposability properties. Using these they verify the
  conjecture for products of posets admitting a symmetric chain
  decomposition or a skew chain partition. This extends investigations
  of the largest Whitney numbers, i.e., width, of such products
  in~\cite{Gr-77} and~\cite{WeKl-79}.
\item Tovey and West~\cite{Tov-85} deepen the study of the conjecture
  as a dual integrality statement in linear programs using a network
  flow approach.
\item West~\cite{Wes-87} constructs special unichain coverings for
  posets with the nested saturation property. These are used to prove
  the conjecture for products $P_m\times P_m$ where $P_k$ is a member
  from a special family of polyunsaturated posets introduced in~\cite{Wes-86}. 
\item According to the abstract of~\cite{Wu-98}, Wu provides another sufficient
  condition for posets to satisfy the conjecture.
\item Liu and West~\cite{Liu-08} verify three special cases of the
  Saks-West Conjecture:
\begin{itemize}
\item both posets have width at most $2$.
\item both posets have height at most $2$.
\item $P$ is a \emph{weak order} (a.k.a. ranking) and $Q$ is a poset
  of height at most $2$ whose comparability graph has no cycles.
\end{itemize}
\end{itemize}

\noindent
This paper is organized as follows. In the last subsection of this
introduction we connect semiantichains and unichains to independent
sets and clique covers in products of comparability graphs. This is
used to reprove that products of posets of height 2 satisfy the
conjecture. In Section~\ref{sec:const} we study $d$- and
$c$-decomposable posets.  These are used to state conditions on $P$
and $Q$ that make $P\times Q$ satisfy the conjecture. The main result
is Theorem~\ref{thm:main}, it allows us to reproduce known results as
well as contribute new classes satisfying the conjecture.  In
Theorem~\ref{thm:rectangles} we show that a new class of \emph{rectangular} posets
 has the property that whenever $P$ is rectangular, the
conjecture holds for $P\times Q$ with arbitrary $Q$. For negative, in
Section~\ref{sec:exmpl} we provide a counterexample to the Saks-West
Conjecture. In particular we can produce an arbitrary large gap
between the size of a largest semiantichain and the size of a smallest
unichain covering.  In Section~\ref{sec:comm} we comment on some
natural dual versions of the Saks-West Conjecture raised by Trotter
and West~\cite{Tro-87} and conclude with open problems.

\subsection{Products of comparability graphs}\label{ssec:height2} 

For a graph $G$, let $\alpha(G)$ and $\theta(G)$ denote 
the size of the largest independent set in $G$ and the minimum size of
clique covering of~$G$, respectively.  The comparability graph of a poset $(P,\leq)$
is denoted $G_P$. The semiantichain conjecture has a nice
reformulation in terms of products of two comparability graphs, where the product $G\Box G'$ of two graphs has an edge $(v,v')\sim(u,u')$ iff $v=u$ and $v'\sim u'$ or $v \sim u$ and $v'=u'$. Note
that in general $G_{P\times Q} \neq G_P \Box G_Q$ indeed if $u <_P v$
and $x<_Q y$, then $(u,x) <_{P\times Q} (v,y)$ by transitivity, but in
$G_P \Box G_Q$ there is no edge between $(u,x)$ and $(v,y)$. 
In fact unichains in $P\times Q$ and cliques in $G_P \Box G_Q$ are in bijection.
Hence, the following holds: 

\Bitem $\alpha(G_P \Box G_Q)$ equals the size of the largest
semiantichain in $P\times Q$.

\Bitem $\theta(G_P \Box G_Q)$ equals the size of the minimum
unichain decomposition of $P\times Q$.
\medskip

\noindent
Thus we can reformulate Conjecture~\ref{conj:sw} as:
\begin{conjecture}\label{semi-ant-perf-conj}
For any two posets $P$ and $Q$ it holds $\alpha(G_P\Box G_Q) = \theta(G_P\Box G_Q)$. 
\end{conjecture}

Using this version of the conjecture we now show:
\begin{proposition}
For posets $P$ and $Q$ of height at most 2 the conjecture is true. 
\end{proposition}
\begin{proof}
The comparability graph of a poset of height 2 is bipartite. 
Next we observe:

{\Bitem If $G$ and $G'$ are bipartite graphs, then $G\Box G'$ is again bipartite.
\medskip

}

\noindent
Bipartite graphs are perfect, hence in particular
$\alpha(G)=\theta(G)$ for every bipartite graph.
\end{proof}

\noindent {\it Remark:} The identity $\alpha(G)=\theta(G)$ used in the
proof can also be obtained directly from Dilworth's theorem, we only
have to observe that if $G$ is bipartite, then $G=G_P$ for some poset
$P$ and that $\alpha(G)=w(P)$ while $\theta(G)$ equals the minimum
size of a chain decomposition of $P$.

\section{Constructions}\label{sec:const}

In this section we obtain positive results for posets admitting special chain and antichain partitions. Dual to the concept of
$k$-antichain we call a subset of $P$ a $k$-chain if it is the union
of $k$ disjoint chains. Similarly to $d_k(P)$ we denote the size of a
maximal $k$-chain of $P$ by $c_k(P)$ or simply~$c_k$. 
The main tool for our proof is Theorem~\ref{thm:Greene}, which has been
obtained by Greene~\cite{Gre-76b}. The theorem is a common generalization of the Theorem
of Greene-Kleitman and its dual which has also be obtained by Greene~\cite{Gre-76b}.
Both theorems have been generalized in several directions and 
have been reproved using different methods. Surveys have been given by
Greene and Kleitman~\cite{GK-78} and West~\cite{We-85}. A more recent
survey on a generalization to directed graphs is~\cite{Har-06}.

\begin{theorem}\label{thm:Greene}
  For any poset $P$ there exists a partition
  $\lambda^P=\{\lambda^P_1\geq\ldots\geq\lambda^P_{w}\}$ of $|P|$ such
  that $c_k(P)=\lambda^P_1+\ldots+\lambda^P_k$ and
  $d_k(P)=\mu^P_1+\ldots+\mu^P_k$ for each $k$, where $\mu^P$ denotes
  the conjugate to~$\lambda^P$.
\end{theorem}

Following Viennot~\cite{Vie-84} we call the Ferrers diagram of
$\lambda^P$ the \emph{Greene diagram} of $P$, it is denoted $G(P)$.  
A poset $P$ is {\it $d$-decomposable} if it has an antichain
partition $A_1,A_2,\ldots,A_h$ with $|\bigcup_{i=1}^k A_i| = d_k$ for
each $k$. This is, $|A_k|=\mu_k^P$ for all $k$.
Dual to the notion of $d$-decomposability we
call $P$ {\it $c$-decomposable} if it has a chain partition
$C_1,C_2,\ldots,C_w$ with $|\bigcup_{i=1}^k C_i| = c_k$, i.e.,
$|C_k|=\lambda_k^P$ for all $k$. Chain partitions with this
property have been referred to as \emph{completely saturated},
see~\cite{Tov-81, Gri-95}.

\calc_figscale{51}%
\begin{figure}[htb]
\centerline{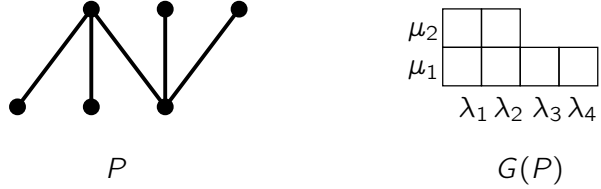}
    \caption{A poset $P$ with its Greene diagram $G(P)$.
Note that $P$ is $c$-decomposable but not $d$-decomposable.\label{fig:ferrers}}
    \end{figure}%


For posets $P$ and $Q$ with families of disjoint antichains
$\{A_1,\ldots, A_k\}$ and $\{B_1,\ldots, B_{\ell}\}$, respectively,
the set $A_1\times B_1 \cup \ldots \cup A_{\min(k,\ell)}\times
B_{\min(k,\ell)}$ is a semiantichain of $P\times Q$. A semiantichain
that can be obtained this way is called \emph{decomposable
  semiantichain}, see~\cite{Tov-81}. By our definitions we have the
following:

\begin{observation}\label{obs:d-dec}
  If $P$ and $Q$ are $d$-decomposable with height $h_P$ and $h_Q$,
  then $P\times Q$ has a decomposable semiantichain of size
$$
\sum_{i=1}^{\min(h_P,h_Q)} \mu_i^P\mu_i^Q.
$$
\end{observation}

In order to construct unichain coverings for $P\times Q$ one can apply
Theorem~\ref{thm:gk} repeatedly.  The resulting coverings are called
\emph{quasi-decomposable} in~\cite{Tov-81}. More precisely:

\begin{proposition}\label{prop:gk} In a product $P\times Q$ where
$\mathcal{C}$ is a chain covering of $P$ there is a unichain covering
of size
$$
\sum_{C\in \mathcal{C}} d_{|C|}(Q).
$$
\end{proposition}
\begin{proof} Use Theorem~\ref{thm:gk} on every $C\times Q$ for
$C\in\mathcal{C}$ to get a unichain covering of size $d_{|C|}(Q)$.  
The union of the resulting unichain coverings is a
unichain covering of $P\times Q$.
\end{proof}

The following theorem has already been noted implicitly by Tovey and
West in~\cite{Tov-81}. 

\begin{theorem}\label{thm:main} If $P$ is $d$-decomposable and
$c$-decomposable and $Q$ is $d$-decomposable, then the size of a
maximum semiantichain and the size of a minimum unichain covering in
the product $P\times Q$ are equal. The size of these is obtained by
the two above constructions, i.e.,
$$
\sum_{i=1}^{\min(h_P,h_Q)} \mu_i^P\mu_i^Q=\sum_{j=1}^{w(P)}
d_{\lambda^P_j}(Q).
$$
\end{theorem}
\begin{proof} Since $P$ and $Q$ are $d$-decomposable, there is a
  semiantichain of size $\sum_{i=1}^{\min(h_P,h_Q)} \mu_i^P\mu_i^Q$ by
  Observation~\ref{obs:d-dec}. On the other hand if we take a chain
  covering $\mathcal{C}$ of $P$ witnessing that $P$ is
  $c$-decomposable we obtain a unichain covering of size
  $\sum_{j=1}^{w(P)} d_{\lambda^P_j}(Q)$ with
  Proposition~\ref{prop:gk}. We have to prove that these values
  coincide. To this end consider the Greene diagrams $G(P)$ and
  $G(Q)$.  Their \emph{merge} $G(P,Q)$ (see Figure~\ref{F:merge}) is
  the set of unit-boxes at coordinates $(i,j,k)$ with $j\leq w(P) =
  \mu^P_1$, $i\leq \min(\lambda^P_j,h_Q)$, and $k\leq\mu_i^Q$.
  Counting the boxes in $G(P,Q)$ by $i$-slices we obtain the left hand
  side of the formula. A given $j$-slice contains
  $\mu^Q_1+\ldots+\mu^Q_{\min(\lambda^P_j,h_Q)}=d_{\lambda^P_j}(Q)$
  boxes. Thus counting the boxes in $G(P,Q)$ by $j$-slices yields the
  right hand side of our formula.  This concludes the proof.
\end{proof}%
\begin{figure}[hbt]%
\input{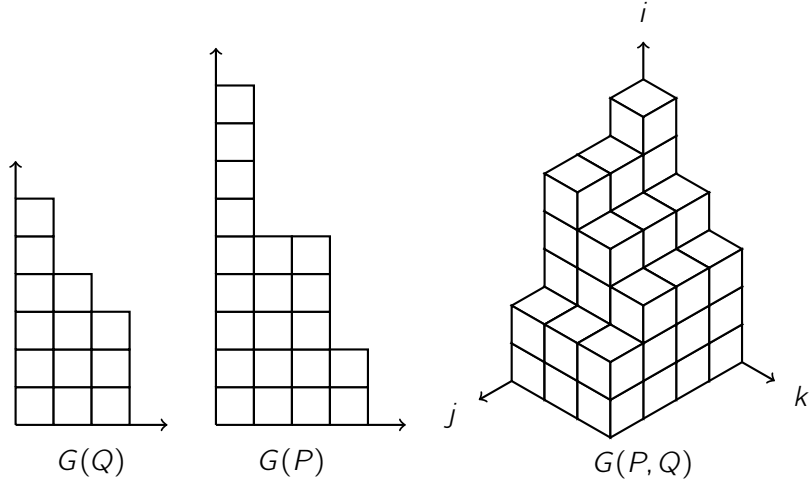}%
\centering{}%
\begin{tikzpicture}[thick, scale=.5]%
\pplanepartition{6,4,3}{$G(Q)$}%
\pgftransformshift{\pgfpoint{150}{0}};
\pplanepartition{9,5,5,2}{$G(P)$}%
\pgftransformshift{\pgfpoint{320}{90}};
\planepartition{{6,4,3},{5,4,3},{5,4,3},{2,2,2}}{$G(P,Q)$}%
\end{tikzpicture}%
\caption{Merge of two Greene diagrams.}%
\label{F:merge}%
\end{figure}
Theorem~\ref{thm:main} includes some interesting cases
for the min-max relation that have been known but also adds a few new
cases. These instances follow from proofs that certain classes of
posets are $d$-decomposable, respectively $c$-decomposable.

A graded poset $P$ whose ranks yield an antichain partition witnessing
that $P$ is $d$-decomposable is called \emph{strongly Sperner},
see~\cite{Gri-80}. For emphasis we repeat

\Bitem Strongly Sperner posets are $d$-decomposable.
\medskip

\noindent
For a chain $C$ in $P$ denote by $r(C)$ the set of ranks used by $C$.
A chain-partition $\mathcal{C}$ of $P$ is called \emph{nested} if for
each $C,C'\in\mathcal{C}$ we have $r(C)\subseteq r(C')$ if
$|C|\leq|C'|$. The most examples of nested chain partitions are
symmetric chain partitions. In~\cite{Gri-80} Griggs observes that
nested chain-partitions are completely saturated and that posets
admitting a nested chain partition are strongly Sperner. Hence we
have the following

\Bitem Posets that have a nested chain partition are $d$-decomposable
and $c$-decomposable.
\medskip

\noindent
The fact that products of posets with nested chain partitions satisfy
the Saks-West Conjecture was shown by West~\cite{Wes-87}. A
special class of strongly Sperner posets are \emph{LYM posets}. A
conjecture of Griggs~\cite{Gri-77} that remains
open~\cite{Esc-11,Wan-05} and seems interesting in our context is that
LYM posets are $c$-decomposable.

\Bitem Orders of width at most $3$ are $d$-decomposable.
\begin{proof} Since $P$ has width at most $3$ we have
$\mu_i\in \{1,2,3\}$ for all $i \leq h_P$. Let $a,b,c$
be the numbers of $3$s, $2$s, and $1$s in $\mu_1,\ldots,\mu_k$,
respectively. We will find an antichain partition of $P$ such that $a$
antichains will be of size $3$, $b$ antichains will be of size $2$ and
$c$ antichains will have size $1$. Let $A\subseteq P$ be a maximum
$(a+b)$-antichain. From Greene's Theorem (Thm.~\ref{thm:Greene}) we know
that $|A| = 3a+2b$. Since $|P-A| = c$ we can partition this set into
$c$ antichains of size $1$. Now consider a partition $\bigcup_{i=1}^{h_A} B_i$ of $A$
such that $B_i$ is the set of minimal points in $B_i\cup\ldots\cup
B_{h_A}$. Since $|B_i|\leqslant 3$ we may consider $a',b',c'$ as the
numbers of $3$s, $2$s, and $1$s in all $|B_i|$ (for $i=1,\ldots,
h_A$). With these numbers we have $|A|=3a'+2b'+c'=3a+2b$ and $h_A =
a'+b'+c' = a+b$. Note that $a'\leqslant a$, otherwise we would have 
an $(a+1)$-antichain of size $3(a+1) > d^A_{a+1}$. 
Since $|A|-2h_A = a = a' - c'$ we obtain $c'=0$, $a'=a$ and $b'=b$. 
This completes the proof.
\end{proof}

\Bitem Series-parallel orders are $d$-decomposable and
$c$-decomposable.

\Bitem Weak orders are $d$-decomposable and
$c$-decomposable.
\medskip

\noindent
Since weak orders are a subclass of series-parallel orders
the second item follows from the first which is 
implied by the following lemma.

\begin{lemma}\label{thm:sp-comp+cd-dec} 
If $P$ and $P'$ are  $d$-decomposable (resp. $c$-decomposable),
then the same holds for their series composition $P*Q$
and their parallel composition $P+Q$.
\end{lemma}
\begin{proof} For $d$-decomposability let ${\cal A} = \{A_1,\dots,
  A_{h(P)}\}$ and ${\cal A'} = \{A'_1,\dots, A'_{h(P')}\}$ be
  witnesses for $d$-decomposability of $P$ and $P'$,
  respectively. Ordering the antichains of ${\cal A} \cup {\cal A'}$
  by size yields a witness for $d$-decomposability of $P*P'$.  For the
  parallel composition $P+P'$ note that $(A,A') \longleftrightarrow
  A\cup A'$ is a bijection between pairs of antichains with $A\subseteq
  P$ and $A'\subseteq P'$ and antichains in $P+P'$. Therefore,
  $d_k(P+P')=d_k(P)+d_k(P')$ and the antichain partition
  $\{A_1\cup A'_1,\dots, A_{h}\cup A'_{h}\}$ of $P+P'$
  proves $d$-decomposability (we let $h=\max(h(P),h(P'))$ and use 
  empty antichains for indices exceeding the height).

  For $c$-decomposability the same proof applies but with roles changed
  between $P+P'$ and $P*P'$, i.e., ordering the chains of ${\cal C}
  \cup {\cal C'}$ by size yields a witness for $c$-decomposability of
  $P+P'$ while $c$-decomposability of $P*P'$ is witnessed by the chain
  partition $\{C_1\cup C'_1,\dots, C_{w}\cup C'_{w}\}$.
\end{proof}

With the next result we provide a rather general extension of Theorem~\ref{thm:gk}, i.e., we exhibit a class of posets such that
every product with one of the factors from the class satisfies the
conjecture.

A poset $P$ is \emph{rectangular} if $P$ contains a poset
$L$ consisting of the disjoint union of $w$ chains of length $h$ and
$P$ is contained in a weak order $U$ of height $h$ with levels of size
$w$. Here containment is meant as an inclusion among binary relations.
Note that since they may contain maximal chains of size $< h$
rectangular posets need not be graded. Still there is natural concept
of rank and of a nested chain-decomposition, these can be used to show
that rectangular posets are $c$-and $d$-decomposable.  Even more
can be said:

\begin{theorem}\label{thm:rectangles} In a product $P\times Q$ where $P$
is rectangular of width $w$ and height $h$ the size of a largest
semiantichain equals the size of a smallest unichain
covering. Moreover, this number is $wd_{h}(Q)$.
\end{theorem}
\begin{proof} $P$ contains a poset $L$ consisting of the disjoint
union of $w$ chains of length $h$ and $P$ is contained in a weak order
$U$ of height $h$ with levels of size $w$. Using
Proposition~\ref{prop:gk} we obtain a unichain covering of $L\times Q$
of size $\sum_{i=1}^w d_{h}(Q)=w d_h(Q)$. This unichain covering is also a unichain covering of $P\times Q$
of the required size. On
the other hand in $U\times Q$ we can find a decomposable semiantichain
as a product of the ranks of $U$ with the antichain decomposition
$B_1,\ldots, B_{h}$ of a maximal $h$-antichain
in $Q$. The size of this semiantichain is then
$\sum_{i=1}^{h} w|B_i|=w d_h(Q)$ in $U\times
Q$. The semiantichain of $U\times Q$ is also a semiantichain of
$P\times Q$. This concludes the proof.
\end{proof}

\section{A bad example}\label{sec:exmpl} 
To simplify the analysis of the counterexample we use the following property of weak
orders which may be of independent interest.
\begin{proposition}\label{prop:weak} If $P$ is a weak order and $Q$ is
an arbitrary poset, then the maximal size of a semiantichain in
$P\times Q$ can be expressed as $\sum_{i=1}^{k} \mu_i^P\cdot |B_i|$
where $B_1,B_2,\ldots,B_k$ is a family of disjoint antichains in $Q$.
\end{proposition}
\begin{proof} Let $S$ be a semiantichain in $P\times Q$. For any
$X\subseteq P$ denote by $S(X):=\{q\in Q\mid p\in X, (p,q)\in S\}$.
Recall that for any $p\in P$ the set $S(\{p\})$ (or shortly $S(p)$) is
an antichain in $Q$.  Now take a level $A_i=\{p_1,\ldots,p_{k}\}$ of
$P$ and let $B_i$ be a maximum antichain among
$S(p_1),\ldots,S(p_{k})$.  Replacing $\{p_1\}\times
S(p_1),\ldots,\{p_{k}\}\times S(p_{k})$ in $S$ by $A_i\times B_i$ we
obtain $S'$ with $|S'|\geq|S|$.  Moreover, since $P$ is a weak order
the $S(A_i)$ are mutually disjoint. This remains true in $S'$. Thus
$S'$ is a semiantichain. Applying this operation level by level we
construct a decomposable semiantichain of the desired size.
\end{proof}

\calc_figscale{65}%
\begin{figure}[htb]
\centerline{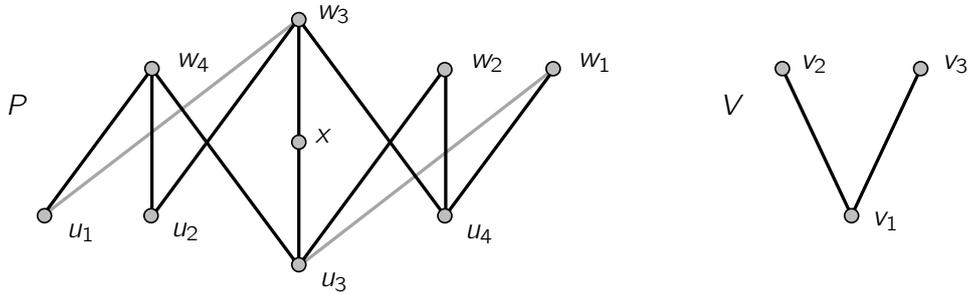}
    \caption{A pair $(P,V)$ of posets disproving the
  conjecture.  The comparabilities depicted in gray are optional. The
  argument does not depend on whether they belong to $P$ or not.\label{fig:counter-small}}
    \end{figure}%

\label{fig:bad-exm}

Let $P$ and $V$ be the posets shown on in
Figure~\ref{fig:counter-small}.  Since $V$ is a weak-order we can use
Proposition~\ref{prop:weak} to determine the size of a maximum
semiantichain in $P\times V$. We just need to maximize formula
$2a+b$ where $a$ and $b$ are the sizes of non-intersecting antichains
in $P$. Starting with the observation that there is a unique antichain
of size 5 in $P$ it is easily seen that the optimal value is~$12$ and can be attained
as $2\cdot 5 + 2$ or as $2\cdot4+4$.
We focus on the following maximum semiantichain:
$$
S= \{u_1,u_2,u_3,u_4\}\times\{v_2,v_3\} \cup \{w_1,w_2,w_3,w_4\}\times\{v_1\}
$$

If there is an optimal unichain covering of size 12 then every
unichain has to contain one element an element of $S$.
This implies that the points $(x,v_i)$  have to be
in three different unichains. These three unichains can cover
all points $(u_3,v_i)$ and $(w_3,v_i)$ from $S$.
To cover  $\{u_1,u_2,w_4\}\times V$ we need at least $5$ unichains.
Another 5 unichains are needed for $\{w_1,w_2,u_4\}\times V$. Therefore, 
any cover of $P\times V$ consists of at least $13$ unichains.
\medskip

The above construction can be modified to make the gap between a
maximum semiantichain and a minimum unichain covering arbitrary large.  To see that, just replace $V$ by a height $2$ weak order $V'$
with $k$ minima and $k+1$ maxima. Now consider $P'$ arising from $P$
by blowing up the antichains $\{u_1,u_2\}$ and $\{w_1,w_2\}$ to
antichains of size $k+1$ (again edges between elements from the
$u$-antichain and $w_3$ as well as edges between elements from the
$w$-antichain and $u_3$ are optional). Along the lines of the above
proof the following can be shown:
\begin{remark} The gap between the size of a maximum semiantichain and
a minimum unichain covering in $P'\times V'$ is $k$.
\end{remark} 

Recall that there is no gap if one factor of the product
is rectangular (see Theorem~\ref{thm:rectangles}). The factor $V'$ is \emph{almost}
rectangular but the gap is large.

The class of two-dimensional posets was considered in~\cite{Liu-08} as
the next candidate class for verifying the conjecture. However,
in the above construction both factors are two-dimensional.

In a computer search we have identified the poset $P$ shown in
Figure~\ref{fig:counter-small} (respectively the four
posets that can be obtained by choosing a selection of the gray edges)
as the unique minimal examples for a $P$ such that $P\times V$ 
is a counterexample to the conjecture. One can check that all posets with at most $5$
elements are both $d$-decomposable and
$c$-decomposable. It follows from Theorem~\ref{thm:main} that their products
have semiantichains and unichain decompositions of the same size. Also, clearly a posets of size at most $2$ is either a chain or an antichain.
Hence we obtain:

\begin{corollary}
The four examples for $P$ represented by the left part of Figure~\ref{fig:counter-small}
together with $V$ and the dual $V^d$ are the only counterexample to the conjecture
with at most 27 elements.
\end{corollary}

In the conference version of the paper we have been using a poset $P$
with 13 elements to show that the conjecture fails for $P\times V$.

\section{Further comments}\label{sec:comm}
\subsection{A question of Trotter and West} 

In~\cite{Tro-87} it is shown that the size of a minimum semichain
covering equals the size of a largest uniantichain. They continue to
define a \emph{uniantichain} as an antichain in $P\times Q$ in which
one of the coordinates is fixed, and a \emph{semichain} is a set
$T\subseteq P\times Q$ such that no two distinct elements of $T$ are
contained in an uniantichain. They state the open problem whether the
size of a minimum uniantichain covering always equals the size of a
largest semichain. Recall the version of the Saks-West conjecture
stated as Conjecture~\ref{semi-ant-perf-conj}. A uniantichain is a
clique in $\ovl{G_P}\Box\ovl{G_Q}$ where $\ovl{G_s}$ denotes the
complement of the comparability graph of poset $S$. The question about
minimum uniantichain coverings and maximal semichains in products
therefore translates to the question whether
$$\alpha(\ovl{G_P}\Box\ovl{G_Q}) = \theta(\ovl{G_P}\Box\ovl{G_Q})$$
holds for all posets $P$ and $Q$. We have:

\Bitem This duality does not hold in general.

\begin{proof}
  Let $P^*$ and $V^*$ be the two-dimensional conjugates of $P$ and $V$
  from the previous section. The definition of conjugates implies
  $\ovl{G_{P^*}} = \ovl{G_P}$ and $\ovl{G_{V^*}} = \ovl{G_V}$, hence,
  $P^* \times V^*$ is an example of a product where the cover
  requires more uniantichains than the size of a maximum semichain.
\end{proof}
 
Our positive results may also be ``dualized'' to provide conditions
for equality and classes of posets where the size of a minimum
uniantichain covering always equals the size of a largest
semichain. We exemplify this by stating the dual of
Theorem~\ref{thm:main}:

\Bitem If $P$ is $c$-decomposable and
$d$-decomposable and $Q$ is $c$-decomposable, then the size of a
maximum semichain and the size of a minimum uniantichain covering in
the product $P\times Q$ are equal. The size of these is 
$$
\sum_{i=1}^{\min(w_P,w_Q)} \lambda_i^P\lambda_i^Q
=\sum_{j=1}^{h(P)} c_{\mu^P_j}(Q).
$$

\subsection{Open problems} 
In the present paper the concept of $d$-decomposability is of some
importance. This notion is also quite natural in the context of
Greene-Kleitman Theory. We wonder if there is any ``nice''
characterization of $d$-decomposable posets. Let $P$ be the
six-element poset on $\{x_1,x_2,x_3, y_1,y_2,y_3\}$ with $x_i\leqslant y_2$
and $x_2\leqslant y_i$ for all $i=1,2,3$. Is it true that any poset
excluding $P$ as an induced subposet is $d$-decomposable?

Another set of questions arises when considering complexity
issues. How hard are the optimization problems of determining the size
of a largest semiantichain or a smallest unichain covering for a given
$P\times Q$? What is the complexity of deciding whether $P\times Q$
satisfies the Saks-West Conjecture?
In particular, what can be said about the above questions in the case that $Q$ is the poset $V$ from Figure~\ref{fig:bad-exm}.

\section{Acknowledgments} We are grateful to Tom Trotter for showing
us the Saks-West Conjecture and for many fruitful discussions.

\bibliography{sulit} \bibliographystyle{my-siam}

\begin{thebibliography}{10}

\bibitem{Har-06}
{\sc I.~Ben-Arroyo~Hartman}, {\em Berge's conjecture on directed path
  partitions---a survey}, Discrete Math., 306 (2006), pp.~2498--2514.

\bibitem{Dil-50}
{\sc R.~P. Dilworth}, {\em A decomposition theorem for partially ordered sets},
  Ann. of Math. (2), 51 (1950), pp.~161--166.

\bibitem{Esc-11}
{\sc E.~Escamilla, A.~Nicolae, P.~Salerno, S.~Shahriari, and J.~Tirrell}, {\em
  On nested chain decompositions of normalized matching posets of rank 3},
  Order, 28 (2011), pp.~357--373.

\bibitem{Gre-76b}
{\sc C.~Greene}, {\em Some partitions associated with a partially ordered set},
  J. Combinatorial Theory Ser. A, 20 (1976), pp.~69--79.

\bibitem{Gre-76a}
{\sc C.~Greene and D.~J. Kleitman}, {\em The structure of {S}perner
  {$k$}-families}, J. Combinatorial Theory Ser. A, 20 (1976), pp.~41--68.

\bibitem{GK-78}
{\sc C.~Greene and D.~J. Kleitman}, {\em Proof techniques in the theory of
  finite sets}, in Studies in combinatorics, vol.~17 of MAA Stud. Math., Math.
  Assoc. America, Washington, D.C., 1978, pp.~22--79.

\bibitem{Gri-77}
{\sc J.~R. Griggs}, {\em Sufficient conditions for a symmetric chain order},
  SIAM J. Appl. Math., 32 (1977), pp.~807--809.

\bibitem{Gr-77}
{\sc J.~R. Griggs}, {\em Symmetric Chain Orders, Sperner Theorems, and Loop
  Matchings}, PhD thesis, MIT, 1977.

\bibitem{Gri-80}
{\sc J.~R. Griggs}, {\em On chains and {S}perner {$k$}-families in ranked
  posets}, J. Combin. Theory Ser. A, 28 (1980), pp.~156--168.

\bibitem{Gri-95}
{\sc J.~R. Griggs}, {\em Matchings, cutsets, and chain partitions in graded
  posets}, Discrete Math., 144 (1995), pp.~33--46.

\bibitem{Liu-08}
{\sc Q.~Liu and D.~B. West}, {\em Duality for semiantichains and unichain
  coverings in products of special posets}, Order, 25 (2008), pp.~359--367.

\bibitem{Sak-79}
{\sc M.~Saks}, {\em A short proof of the existence of {$k$}-saturated
  partitions of partially ordered sets}, Adv. in Math., 33 (1979),
  pp.~207--211.

\bibitem{Tov-81}
{\sc C.~A. Tovey and D.~B. West}, {\em Semiantichains and unichain coverings in
  direct products of partial orders}, SIAM J. Algebraic Discrete Methods, 2
  (1981), pp.~295--305.

\bibitem{Tov-85}
{\sc C.~A. Tovey and D.~B. West}, {\em Networks and chain coverings in partial
  orders and their products}, Order, 2 (1985), pp.~49--60.

\bibitem{Tro-87}
{\sc L.~E. Trotter, Jr. and D.~B. West}, {\em Two easy duality theorems for
  product partial orders}, Discrete Appl. Math., 16 (1987), pp.~283--286.

\bibitem{Vie-84}
{\sc X.~G. Viennot}, {\em Chain and antichain families, grids and {Y}oung
  tableaux}, in Orders: description and roles ({L}'{A}rbresle, 1982), vol.~99
  of North-Holland Math. Stud., North-Holland, Amsterdam, 1984, pp.~409--463.

\bibitem{Wan-05}
{\sc Y.~Wang}, {\em Nested chain partitions of {LYM} posets}, Discrete Appl.
  Math., 145 (2005), pp.~493--497.

\bibitem{We-85}
{\sc D.~B. West}, {\em Parameters of partial orders and graphs: packing,
  covering, and representation}, in Graphs and Order, I.~Rival, ed., D. Reidel
  Dordrecht, 1985, pp.~267--350.

\bibitem{Wes-86}
{\sc D.~B. West}, {\em Poly-unsaturated posets: the {G}reene-{K}leitman theorem
  is best possible}, Journal of {C}ombinatorial {T}heory. Series A, 41 (1986),
  pp.~105--116.

\bibitem{Wes-87}
{\sc D.~B. West}, {\em Unichain coverings in partial orders with the nested
  saturation property}, Discrete Math., 63 (1987), pp.~297--303.

\bibitem{WeKl-79}
{\sc D.~B. West and D.~J. Kleitman}, {\em Skew chain orders and sets of
  rectangles}, Discrete Math., 27 (1979), pp.~99--102.

\bibitem{WeSa-82}
{\sc D.~B. West and M.~Saks}, {\em Research problem 10}, Discrete Math., 38
  (1982), p.~126.

\bibitem{Wu-98}
{\sc C.~Wu}, {\em On relationships between semi-antichains and unichain
  coverings in discrete mathematics}, Chinese Quart. J. Math., 13 (1998),
  pp.~44--48.

\end{thebibliography}

\end{document}